\documentclass{article}
\usepackage{amssymb,amsfonts,amsmath,amsthm,subfigure}
\usepackage{authblk}
\usepackage[percent]{overpic}
\usepackage{txfonts}
\usepackage{graphicx}

\numberwithin{equation}{section} 

\newtheorem{thm}{Theorem}[section]
\newtheorem{defn}[thm]{Definition}
\newtheorem{cor}[thm]{Corollary}
\newtheorem{lemma}[thm]{Lemma}

\newtheorem{rmrk}[thm]{Remark}

\newcommand{\e}{\varepsilon}
\newcommand{\R}{\mathbb{R}}

\newcommand{\J}{\mathbb{J}}
\newtheorem{prop}[thm]{Proposition}

\newcommand{\abs}[1]{\left\vert{#1}\right\vert}

\newcommand{\ba}{\begin{array}}
\newcommand{\ea}{\end{array}}

\newcommand{\bthm}{\begin{thm}}
\newcommand{\ethm}{\end{thm}}
\newcommand{\bstp}{\begin{stp}}
\newcommand{\estp}{\end{stp}}
\newcommand{\blemma}{\begin{lemma}}
\newcommand{\elemma}{\end{lemma}}
\newcommand{\bprop}{\begin{prop}}
\newcommand{\eprop}{\end{prop}}
\newcommand{\bpf}{\begin{pf}}
\newcommand{\mA}{\mathcal{A}}
\newcommand{\epf}{\end{pf}}
\newcommand{\bdefn}{\begin{defn}}

\newcommand{\edefn}{\end{defn}}
\newcommand{\brk}{\begin{rmrk}}
\newcommand{\erk}{\end{rmrk}}
\newcommand{\bcrl}{\begin{crl}}
\newcommand{\ecrl}{\end{crl}}
\newcommand{\norm}[1]{\left\|#1\right\|}

\newcommand{\beqn}{\begin{equation}}
\newcommand{\eeqn}{\end{equation}}

\newcommand{\gb}{\gamma_{\beta}}

\renewcommand{\leq}{\leqslant}
\renewcommand{\geq}{\geqslant}

\newcommand{\lm}{\lambda}

\newcommand{\mG}{\mathcal{G}}

\newcommand{\mP}{\mathcal{P}}

\newcommand{\beq}{\begin{equation}}
\newcommand{\eeq}{\end{equation}}
\newcommand{\bea}{\begin{eqnarray}}

\newcommand{\pmi}{{\bf p}_-}
\newcommand{\ppl}{{\bf p}_+}
\newcommand{\pz}{{\bf p}_0}

\newcommand{\eea}{\end{eqnarray}}

\begin{document}
\renewcommand\Authfont{\small}
\renewcommand\Affilfont{\itshape\footnotesize}

\author[1]{Jiri Dadok\footnote{dadok@indiana.edu}}
\author[2]{Peter Sternberg\footnote{sternber@indiana.edu}}
\affil[1,2]{Department of Mathematics, Indiana University, Bloomington, IN 47405}
\title{A Degenerate Isoperimetric Problem in the Plane}

\maketitle

\noindent {\bf Abstract:} We establish sufficient conditions for existence of curves minimizing length as measured with respect to a degenerate metric on the plane while enclosing a specified amount of Euclidean area. Non-existence of minimizers can occur and examples are provided. This continues the investigation begun in \cite{ABCDS} where the metric $ds^2$ near the singularities equals a quadratic polynomial times the standard metric. Here we allow the conformal factor to be any smooth non-negative potential
vanishing at isolated points {\em provided} the Hessian at these points is positive definite. These isoperimetric curves, appropriately parametrized, arise as traveling wave solutions to a bi-stable Hamiltonian system.
\noindent


\date


\section{Introduction}
Given a complete metric on $\R^2$ conformal to the standard Euclidean metric, it is not hard to show that
for any two points $p$ and $q$ in $\R^2$ and any positive number $A$, one can find a shortest curve $\gamma$ joining $p$ and $q$
among all competitors which, together with a distance minimizing geodesic $\gamma_0$ from $q$ to $p$ enclose a specified amount of {\em Euclidean} area. 
This variant of an isoperimetric problem, in which length and area are measured with respect to different metrics, becomes more subtle, however, when the conformal factor, $F$, is allowed to vanish at one or more points in the plane. 
 
 Motivated by the observation that such isoperimetric curves can be used to build traveling wave solutions to certain Hamiltonian systems, the authors of \cite{ABCDS} analyze this problem for the case where $W:=F^2$ is a non-negative map from $\R^2$ to $\R$ vanishing at two points, say $\pmi$ and $\ppl$, and where the isoperimetric curve is required to join these two potential wells. Thus, the problem takes the form
 \beq
 \inf\bigg\{\int_a^b F(\gamma(t))\abs{\gamma'(t)}\,dt:\;\gamma(a)=\pmi,\,\gamma(b)=\ppl,\;\int_{\gamma}\omega_0=A\bigg\},
 \label{mp}
 \eeq
 where $\omega_0$ is the $1$-form $p_1dp_2$, and so the interpretation as a Euclidean area constraint comes through an application of Stokes Theorem to the closed curve $\gamma\cup \gamma_0$ in that $d\omega_0=dp_1\,dp_2.$ An equivalent version of the problem results from replacing $\omega_0$ by any $1$-form $\omega$ such that $d\omega=dp_1\,dp_2.$
 
 What makes this particular isoperimetric
problem non-standard is both the degeneracy of the conformal factor and the fact that length is measured with respect to a metric given by $F$
while area is measured with respect to the Euclidean metric. There is a vast literature on isoperimetric problems and in particular isoperimetric curves with assorted assumptions on the conformal factor or ``density,"
though to our knowledge none address this combination of degeneracy and mixture of metrics. We mention, for example, \cite{CJQW,CMV,DDNT,H,RCBM} but of course there are many others.

 Under the assumption that $W$ is exactly equal to a nondegenerate quadratic polynomial in neighborhoods of the zeros ${\bf p}_\pm$, such an isoperimetric curve joining $\pmi$ to $\ppl$ is shown to exist in \cite{ABCDS} for any amount of Euclidean area $A$ via a calibration argument. Near the wells, these minimizers are realized as integral curves of an explicit vector field and interestingly one finds that they often spiral into the zeros of $F$. On the other hand, if $F(p)/\abs{p-{\bf p}_\pm}\to 0$ as $\abs{p-{\bf p}_\pm}\to 0$ for either zero, then minimizers do not exist for any $A\not =\int_{\gamma_0}\omega_0.$
 
In this article we relax the rather stringent assumption that $W$ is a quadratic near its isolated zeros and instead simply assume that it is smooth and has positive definite
Hessian matrix $D^2W$ at its two zeros, along with an assumption on behavior at infinity to ensure completeness of the metric. Since near either potential well $W$ is well-approximated by its quadratic Taylor polynomial it would seem reasonable to conjecture that the same type of existence result as the one found in \cite{ABCDS} is true in this more general setting. Surprisingly this turns out to be false. 

Our main result, Theorem \ref{mainresult}, states than when $A$ is sufficiently close to $\int_{\gamma_0}\omega_0$, an isoperimetric curve exists. However, in Section 4 we reveal through some counter-examples that if $\abs{A-\int_{\gamma_0}\omega_0}$ is not small, a minimizer may fail to exist. The mechanism for non-existence is associated with minimizing sequences packing part of the area around the wells. In the limit this part `disappears' and the limiting curve encloses a smaller amount of area. 

The existence result relies on the introduction of a complete Carnot-Carath\'eodory length metric on $\R^3$ (cf. \cite{Gr}) with singularities over the wells. The geodesics for this metric project to the plane as solutions to our problem. This re-phrasing of the problem is explained in Section 2 but roughly speaking, the ``extra" component of competing curves keeps track of the amount of Euclidean area accumulated by a planar curve as it journeys from a point $p$ to a point $q$.

We review in Section 3 the results from \cite{ABCDS} on the case of quadratic $W$, along with some essential extensions. Section 4 contains the counter-examples, Section 5 contains the proof of the main existence result and in Section 6 we recall from \cite{ABCDS} how under further non-degeneracy assumptions, these isoperimetric curves can serve as traveling wave solutions to the Hamiltonian
system
\beq
\J u_t=\Delta u -\nabla W_u(u)\quad\mbox{for}\;u:\R^n\times\R\to\R^2\label{funnyHam}
\eeq
where  $\J$ denotes the symplectic matrix given by
\[
\J=\left(\begin{matrix} 0 & 1\\ -1& 0\end{matrix}\right).
\]
This not yet well-understood system of partial differential equations represents a form of conservative dynamics associated with the well-studied vector Allen-Cahn (or vector Modica-Mortola) energy
\[
u\mapsto\int \frac{1}{2}\abs{\nabla u}^2 + W(u),
\]
and could provide an interesting model for phase transitions in the Hamiltonian setting. One suspects that having a solid understanding of traveling wave solutions will be a helpful first step in subsequent analysis of the dynamics. Invoking the traveling wave ansatz 
\[u(x_1,\ldots,x_n,t)=U(x_1-\nu t)\] in \eqref{funnyHam}, one sees that $U:\R\to\R^2$ must solve
 the system of ODE's
\beq
-\nu \J U'=U''-\nabla_uW(U)\quad\mbox{for}\;-\infty<y<\infty,\quad U(\pm\infty)={\bf p}_{\pm}.
\label{introODE}
\eeq
It turns out that isoperimetric curves $\gamma$ solving \eqref{mp}, appropriately parametrized, will solve this system where the wave speed $\nu$ depends on the area constraint $A$ and can also be related to the geodesic curvature $\kappa_g$ of $\gamma$
through the formula $\nu=\kappa_g W(\gamma)$, a quantity that turns out to be constant along the isoperimetric curves, cf. Remark \ref{gc}.

\vskip.1in

\noindent
{\bf Acknowledgment.} The authors wish to thank Stan Alama, Lia Bronsard and Andres Contreras for many helpful conversations on this problem.
P.S. also wishes to acknowledge the support of the National Science Foundation through D.M.S. 1362879.
\section{A re-casting of the problem in terms of curves in $3$-space}

We begin by restating our hypotheses more precisely. The arguments we will present apply to the setting where the conformal factor vanishes 
at any set of isolated points but for ease of presentation we will take the zero set to be two points.
Thus, we assume throughout that $W:\R^2\to[0,\infty)$ is a smooth function vanishing only at two points $\pmi$ and $\ppl$ with positive definite
Hessian matrix $D^2W({\bf p}_\pm)>0$. We also assume
\[
\liminf_{\abs{p}\to\infty}W(p)>0.
\]
Then we define $F:\R^2\to [0,\infty)$ by
\[F(p):=\sqrt{W(p)}\].

We first recall from \cite{ZS}  the following result (see also \cite{MS}):\\
\bthm\label{ungeo}
For any two points $p$ and $q$ in $\R^2$ there exists a (not necessarily unique) minimizer to the problem
\beq
\overline{d}_F(p,q):=\inf\left\{E(\gamma):\;\gamma:[a,b]\to\R^2\;\mbox{is locally Lipschitz continuous and}\;\gamma(a)=p,\;\gamma(b)=q\right\}.\label{geo}
\eeq
\ethm
For the case $p=\pmi$ and $q=\ppl$ we choose $\gamma_0$ as a minimizer to \eqref{geo} and we refer to such a curve
as an {\it unconstrained planar geodesic} to distinguish it from curves solving the constrained problem \eqref{mainprob} below and from certain geodesics in three dimensions to be introduced shortly.

For any locally Lipschitz continuous mapping $\gamma:[a,b]\to\R^2$, writing $\gamma(t)=(\gamma^{(1)}(t),\gamma^{(2)}(t))$ we now define
\beq
E(\gamma):=\int_a^b F(\gamma(t))\abs{\gamma'(t)}\,dt\quad\mbox{and}\quad
\mA(\gamma):=\int_{\gamma}\omega_0=\int_a^b\gamma^{(1)}(t)\,\gamma^{(2)}\,'(t)\,dt,\label{Edefn}
\eeq
where $\omega_0=p_1dp_2$.
Our goal is to establish the existence of a solution to the problem
\begin{eqnarray}
&&m_A:=\nonumber\\
&&\inf\left\{E(\gamma):\;\gamma:[a,b]\to\R^2\;\mbox{is locally Lipschitz continuous},\;\gamma(a)=\pmi,\;\gamma(b)=\ppl,\;\mA(\gamma)=A \right\}\nonumber\\ \label{mainprob}
\end{eqnarray}
for certain values of the area constraint $A$.

It will be convenient in what follows to introduce notation for the projection from $\R^3$ down to $\R^2$ so we will write
$\Pi(P):=(p_1,p_2)$ for any point $P=(p_1,p_2,p_3)\in \R^3.$

We also introduce a certain lifting of a planar curve into $\R^3$. To any locally
Lipschitz continuous planar curve $\gamma:[a,b]\to\R^2$ we will frequently associate
a locally Lipschitz curve in $3$-space $\Gamma:[a,b]\to\R^3$ via the properties $\Pi(\Gamma)=\gamma$ so that $\Gamma^{(j)}=\gamma^{(j)}$ for $j=1,2$ and
\beq
\Gamma^{(3)}\,'(t)=\gamma^{(1)}(t)\,\gamma^{(2)}\,'(t)\;\mbox{for}\;a.e.\;t\in (a,b).
\label{assoc}
\eeq
We point out that \eqref{assoc} only determines the third component function of the curve $\Gamma$ up to an arbitrary additive constant and that
\[
\Gamma^{(3)}(b)-\Gamma^{(3)}(a)=\mA(\gamma).
\]
With this notation in hand we define the singular Carnot-Carath\'eodory length metric $d_F:\R^3\times \R^3\to [0,\infty)$ (cf. \cite{Gr}) via
\beq
d_F(P,Q):=\inf\Big\{E(\gamma):\;\Pi(\Gamma)=\gamma,\;\Gamma(a)=P,\;\Gamma(b)=Q,\;\Gamma\;\mbox{satisfies}\;\eqref{assoc}\Big\}.\label{ddefn}
\eeq
From now on, when a curve in three-dimensions satisfies \eqref{assoc} we will say it is {\em admissible} in \eqref{ddefn}
and we comment that for any two points $P=(p_1,p_2,p_3)$ and $Q=(q_1,q_2,q_3)$, the quantity $d_F(P,Q)$ only depends on $p_3$ and $q_3$ through the difference $q_3-p_3$. One can also readily check through the association of two-dimensional curves and admissible three-dimensional curves via lifting that an equivalent characterization of $d_F$ is
\beq
d_F(P,Q)=\inf\Big\{E(\gamma):\;\gamma(a)=(p_1,p_2),\;\gamma(b)=(q_1,q_2),\;\int_a^b\gamma^{(1)}(t)\gamma^{(2)}\,'(t)\,dt=q_3-p_3\Big\}.
\label{eq}
\eeq
One sees immediately from a comparison with \eqref{mainprob} that when $\Pi(P)=\pmi$, $\Pi(Q)=\ppl$ and $q_3-p_3=A$, one has
\[
d_F(P,Q)=m_A.
\]

The advantage of this re-formulation of problem \eqref{mainprob} comes through the introduction of the
notion of length for {\em any} curve $\Gamma:[a,b]\to\R^3$, not necessarily one satisfying \eqref{assoc}, via
\beq
L(\Gamma):=\sup_{\{t_j\}^N_{j=1}\in \mP([a,b])} \sum^N_{j=1} d_F(\Gamma(t_j),\Gamma(t_{j+1})),\label{Ldefn}
\eeq
where $\mP([a,b])$ is the set of finite partitions of $[a,b]$.
As we will see below, there {\it always} exists a curve $\Gamma$ in $\R^3$ of least length--though not necessarily admissible!-- with endpoints $P$ and $Q$ satisfying the property $L(\Gamma)=d(P,Q)$ whereas solutions to \eqref{mainprob} sometimes do not exist.

To understand better the relation between the metric $d_F$ and the length of curves we first observe:
\bprop \label{preGromov} The metric $d_F$ given in \eqref{ddefn} defines a length metric on $\R^3$ (cf. \cite{Gr}); that is, for any points $P$ and $Q$ in $\R^3$ one has
\beq
d_F(P,Q)=\inf_{\Gamma} L(\Gamma)\label{dL}
\eeq
where the infimum on the right is taken over all curves $\Gamma:[a,b]\to\R^3$ that are Lipschitz continuous with respect to the metric $d$ and satisfy
$\Gamma(a)=P,\;\Gamma(b)=Q$ and $\Gamma^{(3)}\,'(t)=\Gamma^{(1)}(t)\,\Gamma^{(2)}\,'(t)$ for a.e. $t\in (a,b)$, thus making the metric space into a length space, cf. \cite{bridson}.

\begin{proof}
Away from any zeros of $F$ the metric property of $d_F$ is standard but there is one subtlety to check: Given two distinct points $P$ and $Q$
that both project to either $\pmi$ or $\ppl$, we must verify that $d_F(P,Q)\not=0.$  This follows from the assumption of non-degeneracy of the
Hessian matrices $D^2W({\bf p}_\pm)>0$, since locally, for instance near $\pmi$, one can fit a non-degenerate quadratic polynomial under $W$, say 
$\tilde{W}(p)=\tilde{\lm_1}^2p_1^2+\tilde{\lm_2}^2p_2^2$, so that 
\[
F(p):=\sqrt{W(p)}\geq \tilde{F}(p):=\sqrt{\tilde{\lm_1}^2p_1^2+\tilde{\lm_2}^2p_2^2}.
\]
It follows that $d_F(P,Q)\geq d_{\tilde{F}}(P,Q)$ and the behavior of the metric for such homogeneous conformal factors was completely
worked out in \cite{ABCDS}. The relevant results are all reviewed in Section 3 to follow but in particular through an application of Corollary \ref{FHvert} we have that
\[
d_F(P,Q)=d_F\big((\pmi,0),(\pmi,q_3-p_3)\big)\geq d_{\tilde{F}}\big((\pmi,0),(\pmi,q_3-p_3)\big)=(\tilde{\lm}_1+\tilde{\lm}_2)\abs{q_3-p_3}>0
\]
in case $P$ and $Q$ both project down to $\pmi$ with a similar inequality holding if $\Pi(P)=\Pi(Q)=\ppl.$

To establish the equivalence \eqref{dL} we can use the fact that 
\beq
L(\Gamma)=E\big(\Pi(\Gamma)\big)\label{LE}
\eeq
 for any admissible $\Gamma$ which follows by the same argument used to prove Theorem 2.5 of \cite{ZS}. Taking the infimum of both sides leads to the result.
\end{proof}
\eprop

\brk We wish to emphasize the crucial role that the non-degeneracy of the Hessian matrices $D^2W({\bf p}_\pm)$ plays here. For example,
suppose $\pmi=(0,0)$ and in a neighborhood of the origin $W$ takes the form of a radial function $W=W(r)=r^\alpha$ for $\alpha>2.$ Then
take $P=(0,0,0)$ and $Q=(0,0,A)$ for any $A>0$. With $r_j$ chosen so that 
$j\pi r_j^2=A$, the sequence of planar curves $\{\gamma_j\}$ consisting of 
the directed line segment $\ell_j$ joining $(0,0)$ to $(0,r_j)$, followed by $j$ circles of radius  $r_j$ and ending with $-\ell_j$ returning to the origin has the property that $E(\gamma_j)\to 0$ while $\mathcal{A}(\gamma_j)=A$. In light of \eqref{eq}, this shows that $d_F\big((0,0,0),(0,0,A)\big)=0$
when such a degeneracy is allowed and so we lose the metric  property. See Section 2.2 of \cite{ABCDS} for more discussion.
\erk

The advantage of working in the length space context is the following Hopf-Rinow type result on existence of geodesics. The proof involves a simple application of the Arzela-Ascoli Theorem working in the metric space of $\R^3$ endowed with the $d_F$ metric.
\bthm \label{Gromov} (\cite{bridson}, Prop. 3.7) For every pair of points $P$ and $Q$ in $\R^3$ there exists a minimizing geodesic $\Gamma:[a,b]\to\R^3$ joining $P$ to $Q$ in the sense that  $\Gamma(a)=P,\;
\Gamma(b)=Q$ and $L(\Gamma)=d(P,Q).$ This geodesic, which may not be admissible in the sense of \eqref{assoc}, can be realized as the uniform limit in the metric $d_F$ of a minimizing sequence of admissible curves as described in Proposition \ref{preGromov}.
\ethm

Whenever this geodesic contains no vertical line segments over the wells, that is, whenever it projects down to a curve, or union of curves, each of which avoids the wells $\pmi$ and $\ppl$ except perhaps at their endpoints, then
it is elementary to check that the projection minimizes $E$ subject to fixed endpoints and the appropriate constraint value $q_3-p_3$, and this is the content of the next proposition.
On the other hand, we emphasize that it is {\it not} the case that such a geodesic projects to a constrained minimizer of $E$ if it {\it does} include a vertical line segment over either of the wells, since in this case the projection has `lost area' and fails to satisfy the constraint value.

\bprop\label{goodprojection}
For any two points $P$ and $Q$ in $\R^3$, let $\Gamma:[0,1]\to\R^3$ denote a geodesic joining them, as guaranteed by Theorem \ref{Gromov}. Suppose the closed set $\big\{t\in [0,1]:\,\Pi\big(\Gamma(t)\big)\in \{\pmi,\ppl\}\big\}$ contains no intervals. Then $\gamma:=\Pi(\Gamma)$ minimizes $E$ among all locally Lipschitz curves $\zeta$ joining $\Pi(P)$ to $\Pi(Q)$ and satisfying $\mathcal{A}(\zeta)=q_3-p_3.$ In particular, if the vertical fibers over $\pmi$ and $\ppl$ are not geodesics then a solution to \eqref{mainprob} exists for all values of $A$.
\eprop

\begin{proof} Consider the restriction of $\Gamma$ to a parameter interval $(a,b)\subset[0,1]$ such that the projection of $\Gamma\big(a,b)\big)$ misses the wells $\pmi$ and $\ppl.$  Recall from Theorem \ref{Gromov} that $\Gamma$ can be realized as the uniform limit with respect to the metric $d$ of a minimizing sequence of admissible curves $\Gamma_j\:[0,1]\to\R^3$ for $j=1,2,\ldots$ parametrized with constant speed, i.e. curves satisfying \eqref{assoc} as well as the conditions
\beq
E(\gamma_j)\to d_F(P,Q)\;\mbox{as}\;j\to\infty\quad\mbox{and}\quad F(\gamma_j(t))\abs{\gamma_j'(t)}=E(\gamma_j)\;\mbox{for every}\;j,
\label{minseq}
\eeq
where $\gamma_j:=\Pi(\Gamma_j)$. 

The assumption that the projection $\gamma(a,b)$ misses the points $\ppl$ and $\pmi$ implies that for any $\delta>0$, there exists a positive constant $C_\delta$ such that for all $j$ large and for all $t\in [a+\delta,b-\delta]$ one has $F(\gamma_j(t))\geq C_\delta$. Then applying \eqref{minseq} we find that $\sup_{t\in [a+\delta,b-\delta]}\abs{\gamma_j'(t)}$
 is bounded by a constant independent of $j$. Consequently we can apply Arzela-Ascoli to pass to a subsequence (still denoted by $\gamma_j$) such that 
$\gamma_j\to \gamma\;\mbox{on}\;[a+\delta,b-\delta]$ uniformly in the Euclidean metric,  and such that
\[\gamma_j'\rightharpoonup \gamma'\;\mbox{weakly in}\;L^2\big((a+\delta,b-\delta)\big).
\]
Then for any $t_1,\,t_2\in[a+\delta,b-\delta]$ it follows that
\[
\int_{t_1}^{t_2} \gamma^{(1)}\gamma^{(2)}\,'\,dt=\lim_{j\to\infty} \int_{t_1}^{t_2} \gamma_j^{(1)}\gamma_j^{(2)}\,'\,dt=
\lim_{j\to\infty}\int_{t_1}^{t_2} \Gamma_j^{(3)}\,'\,dt=\Gamma^{(3)}(t_2)-\Gamma^{(3)}(t_1).
\]
Since $\delta$ is arbitrary, we conclude that $\Gamma^{(3)}\,'(t)=\gamma^{(1)}(t)\gamma^{(2)}\,'(t)$ for a.e. $t\in (a,b)$, that is, $\Gamma$ is admissible in the sense of \eqref{assoc}. In particular, we note that admissibility precludes the possibility that $\Gamma$ contains any vertical line segments over points in $\R^2$ other than $\pmi$ or $\ppl$. We then invoke the property of length-minimizing geodesics that for any sub-interval 
$[t_1,t_2]\subset [0,1]$, one has that
$L(\Gamma_{|_{[t_1,t_2]}})$ minimizes length when compared to any other curve joining $\Gamma(t_1)$ to $\Gamma(t_2)$. In light of \eqref{LE},
this means that the projection $\gamma_{|_{[a,b]}}$ minimizes $E$ among competitors $\zeta:[a,b]\to\R^2$ joining $\gamma(a)$ to $\gamma(b)$ and 
sharing the same constraint value $\mathcal{A}(\zeta)=\mathcal{A}(\gamma_{|_{[a,b]}})$. Thus,
\beq
d_F\big(\Gamma(a),\Gamma(b)\big)=L(\Gamma_{|_{[a,b]}})=E\left(\gamma_{|_{[a,b]}}\right).\label{EequalsL}
\eeq
Writing $\Gamma$ as a union of such `good' sub-arcs and summing over this union we conclude that \eqref{EequalsL} holds with $[a,b]$ replaced
by $[0,1]$.

Finally, if the metric $d_F$ is such that the vertical fibers over $\pmi$ and $\ppl$ are not geodesic then $\Gamma$ always projects to a solution
of \eqref{mainprob} for any $A$-values since this condition would preclude the possibility that $\Gamma$ contains any vertical segments over the wells. An example of this phenomenon is given in Section 4. 
\end{proof}
\begin{cor}\label{EulerL} For any two points $P$ and $Q$ in $\R^3$, let $\Gamma:[0,1]\to\R^3$ denote a minimizing geodesic joining them, as guaranteed by Theorem \ref{Gromov}. Let $[a,b]\subset [0,1]$ be any interval such that the projection $\gamma:=\Pi(\Gamma)$ avoids $\ppl$ and $\pmi$ for all $t\in (a,b).$ Then $\Gamma$ restricted to any such interval $(a,b)$ is smooth, satisfies the relation \eqref{assoc},
and has a smooth projection that satisfies the system of ODE's
\beq
-\left(F(\gamma)\,\frac{\gamma'}{\abs{\gamma'}}\right)'+\abs{\gamma'}\nabla F(\gamma)=-\lambda\left(\gamma'\right)^\perp\label{Euler}
\eeq
  for some constant $\lambda$. Here $(x,y)^\perp:=(-y,x)$. If $\Gamma$ is expressible as a union of such curves, that is, if $\Gamma$ contains no vertical line segments over $\pmi$ or $\ppl$, then \eqref{Euler} holds along each curve in the union with the same value of $\lm$.
\end{cor}
\begin{proof}
 The argument that $\Gamma$ is admissible along the parameter interval $(a,b)$ follows as in the proof of Theorem \ref{goodprojection}. Hence, as in that proof the length-minimizing property of $\Gamma$ restricted to $[a,b]$ implies the area-constrained $E$-minimality of the projection
 $\gamma_{|_{[a,b]}}$. In other words, we know 
that 
\[
\int_a^bF(\gamma)\abs{\gamma'}\,dt\leq \int_a^bF(\zeta)\abs{\zeta'}\,dt
\]
whenever $\zeta:[a,b]\to\R^2$ satisfies $\zeta(a)=\gamma(a),\;\zeta(b)=\gamma(b)$ and \[\int_a^b\zeta^{(1)}\zeta^{(2)}\,'\,dt=\int_a^b\gamma^{(1)}\gamma^{(2)}\,'\,dt.\] Applying the theory of Lagrange multipliers we conclude that $\gamma$  restricted to $[a,b]$ must satisfy the criticality condition
\[
\delta E(\gamma;\tilde{\gamma})=\lambda \delta \mathcal{A}(\gamma;\tilde{\gamma})\quad\mbox{for any compactly supported variation}\;\tilde{\gamma}
\]
where $\lambda\in\R$. (Here $\delta$ refers to first variation.) Since $\delta \mathcal{A}(\gamma;\tilde{\gamma})=-\int\big(\gamma'\big)^\perp\cdot\tilde{\gamma}$ and since we know that for $t\in (a+\delta,b-\delta)$ one has $F(\gamma)\geq C_\delta>0$, we can apply standard regularity theory for ODE's to conclude that $\gamma$ is smooth and classically solves \eqref{Euler} through a routine computation of $\delta E(\gamma;\tilde{\gamma})$. Necessarily the lifting of $\gamma$ to $\Gamma$ via \eqref{assoc} over this sub-interval is then also smooth. 

Finally we note that if $\Gamma$ consists of more than one such arc then the Lagrange multiplier $\lm$ cannot vary from arc to arc for if it did, one could create a variation through a bump on one arc and a compensating bump on the other arc that preserved the total constraint value and violated the condition of vanishing first variation.

\end{proof}
\brk\label{gc} {\bf Characterization of criticality in terms of geodesic curvature. } A more geometric way to view the condition of criticality for a planar
 curve $\gamma$ is the following. Suppose we parametrize $\gamma$ by degenerate arclength, say $\ell$, so that 
 \[
 \norm{\gamma'(\ell)}_g=\langle \gamma'(\ell),\gamma'(\ell)\rangle^{1/2}_g:=F(\gamma(\ell))\abs{\gamma'(\ell)}=1\;\mbox{ for some parameter interval}\; 0<\ell<L,
 \]
 where our metric $g$ is given by $g_{ij}=F^2\delta_{ij}$ in standard coordinates. Now we consider the flow given by the exponential map
 $\tilde{\gamma}(\ell,\tau)=exp_{\gamma(\ell)}\tau V$ where $V$ is a normal vector field given by say $f\,N$ with $f:[0,L]\to\R$ and $\norm{N}_g=1$
 so that 
 \[
 N(\ell)=\frac{1}{F(\gamma(\ell))}\frac{(\gamma'(\ell))^\perp}{\abs{\gamma'(\ell)}}=\frac{1}{F(\gamma(\ell)}n(\ell),
 \]
 with $n$ denoting the unit normal in the Euclidean metric. Then we can carry out the first variation calculation as 
 \begin{eqnarray}
 \frac{d}{d\tau}_{\tau=0}E(\tilde{\gamma})&&=\frac{d}{d\tau}_{\tau=0}\int_0^L\langle\,\frac{\partial\tilde{\gamma}}{\partial\ell},\frac{\partial\tilde{\gamma}}{\partial\ell}\,\rangle_g^{1/2}\,d\ell\nonumber\\
 &&=\int_0^L\langle \,\nabla_{\gamma'}\gamma',V  \, \rangle_g^{1/2}\,d\ell=\int_0^L\kappa_g f\,d\ell,\label{geodcurv}
 \end{eqnarray}
 where $\kappa_g$ denotes the geodesic curvature of $\gamma$ with respect to the metric $g$.
 To restrict the flow to variations that preserve Euclidean area to leading order we note that $V=\frac{f}{F}n$ so the allowable $f$ are those such that
 \[
 0=\int_0^{s(L)}\frac{f}{F}\,ds=\int_0^L\frac{f}{F^2}\,d\ell,
 \]
 where $s$ denotes Euclidean arclength. Finally, letting $h(\ell):=\frac{f(\ell)}{F^2(\gamma(\ell))}$ we see from \eqref{geodcurv} that criticality with a Euclidean area constraint 
 means that
 \[
 \int_0^LF^2(\gamma(\ell))\kappa_g(\ell)h(\ell)\,d\ell=0\quad\mbox{for all} \;h:[0,L]\to\R\;\mbox{such that}\;\int_0^Lh(\ell)\,d\ell=0.\]
 We conclude that $F^2\kappa_g=\lm$, a constant, along $\gamma$. Of course in minimizing length with fixed area, one expects curvature to arise in the criticality condition but the point here is that the mis-match of metrics leads to the extra factor of $F^2$.  Relating this to \eqref{Euler}
 we can identify an expression for geodesic curvature as 
 \[
 \kappa_g=\frac{1}{F^2(\gamma)}\bigg[\frac{1}{\abs{\gamma'}^2}\bigg(F(\gamma)\frac{\gamma'}{\abs{\gamma'}}\bigg)'-\frac{1}{\abs{\gamma'}}\nabla F(\gamma)\bigg]\cdot
 (\gamma')^\perp
 \]
where $\gamma$ is given by an arbitrary parametrization.
\erk
\section{The case of quadratic $W$}\label{quadratic}

In \cite{ABCDS} the problem \eqref{mainprob} is solved for the case where $W$ is given by a non-degenerate quadratic in a neighborhood
of $\pmi$ and $\ppl$. The key there is to first find the solution for the case where $W$ vanishes only at one well, taken for convenience to be the origin.  Choosing coordinate axes given by the eigenvectors of the Hessian matrix $D^2W(0,0)$, one considers
$W\;(=F^2)$ to be of the form
\beq
W(p_1,p_2)=\lm_1^2 p_1^2+\lm_2^2 p_2^2\label{purequad}
\eeq
for positive constants $\lm_1$ and $\lm_2$. For this special case where $W$ is purely quadratic so that $\sqrt{W}$ is homogeneous of degree one,
we will introduce the notation $F_H$ via
\beq
F_H(p):=\sqrt{\lm_1^2 p_1^2+\lm_2^2 p_2^2}\label{FH}.
\eeq
For the rest of the article, when we compute distance using the metric $d_{F_H}$ defined by replacing a general conformal factor $F$ by $F_H$
in \eqref{ddefn} we will refer to this metric as the {\it homogeneous metric}.

Then the following result is proven:
\bthm\label{quadthm} (cf. Thm 2.5 of \cite{ABCDS}). For any value $A\in\R$ and any point $p_0\in\R^2\setminus\{(0,0)\}$ there exists a unique solution to 
\beq
\inf\left\{E_H(\gamma):\;\gamma:[a,b]\to\R^2\;\mbox{is locally Lipschitz continuous},\;\gamma(a)=p_0,\;\gamma(b)=(0,0)\;\mbox{and}\;\mA(\gamma)=A \right\},\label{quadprob}
\eeq
where
\[E_H(\gamma):=\int_a^bF_H(\gamma)\abs{\gamma'}\,dt.
\]
This minimizer, denoted by $\gb$, is given explicitly as the integral curve of the vector field
\[
V_\beta(p):=(\cos\beta)\,\Theta(p)-(\sin\beta)\,R(p)
\]
that joins $p_0$ to the origin where
\beq
 R(p):=\frac{\nabla\tilde{r}(p)}{F_H(p)^2},\quad\mbox{with}\quad
\tilde{r}(p):=\frac{1}{2}(\lambda_1\,p_1^2+\lambda_2\,p_2^2),\label{tilder}
\eeq
\[
\mbox{and}\quad\Theta(p):=\frac{\left(-\lambda_2 p_2,\lambda_1 p_1\right)}{F_H(p)^2}.
\]
Here $\beta$ is selected so that
\beq
\frac{\tilde{r}(p_0)}{\lambda_1+\lambda_2}\cot\beta=A+C_0,\label{cot}
\eeq
where $C_0$ is an explicit constant depending only on $p_0$ such that $C(p_0)\to 0$ as $\abs{p_0}\to 0.$
In particular, the planar curve $\gb$ avoids the origin except at the endpoint $t=b$ and satisfies a Euclidean arclength bound of the form
\beq
\int_a^b\abs{\gb'}\,dt<C(\lm_1,\lm_2,p_0,A).\label{euclidbd}
\eeq 
\ethm
\begin{cor}\label{quadcor} For $F=F_H=\sqrt{W}$ with $W$ given by \eqref{purequad}, and for $P=(p_0,0)$ and $Q=(0,0,A)$ with $p_0\not=(0,0)$ the unique minimizing
geodesic $\Gamma_\beta=\big(\Gamma_\beta^{(1)},\Gamma_\beta^{(2)},\Gamma_\beta^{(3)}\big)$ joining $P$ to $Q$ is given by 
\[
\big(\Gamma_\beta^{(1)},\Gamma_\beta^{(2)}\big)=\gb\quad\mbox{and}\quad \Gamma_\beta^{(3)}\,'=\gb^{(1)}\gb^{(2)}\,',\;\Gamma_\beta^{(3)}(a)=0.
\]
Furthermore, one has the identities
\beq
L(\Gamma_\beta)=E(\gamma_{\beta})=d_{F_H}(P,Q)=\tilde{r}(p_0)\csc\beta\label{lengthiden}
\eeq
with the angle $\beta\in (-\pi/2,\pi/2)\setminus \{0\}$ depending on $\pz$ and $A$ through \eqref{cot}.

Lastly, if we parametrize $\Gamma$ by $A$ rather than $\beta$, one has
\beq
\frac{d L(\Gamma_A)}{dA}=(\lm_1+\lm_2)\cos{\beta_A}\quad\mbox{where}\;\beta_A\;\mbox{is determined by}\;\eqref{cot}.\label{maxi}
\eeq
 \end{cor}
 \begin{proof}
 This follows immediately from the fact that \eqref{LE} holds for admissible curves which implies that whenever there exists an admissible planar curve that minimizes $E$ then its lifting minimizes the three-dimensional length $L$. The identify \eqref{lengthiden}, phrased in terms of $E(\gamma_\beta)$ is derived within the proof of
 Thm 2.5 of \cite{ABCDS}. Then \eqref{maxi} follows by direct calculation using \eqref{cot} and \eqref{lengthiden}.
 \end{proof}
 The proof of Theorem \ref{quadthm} follows by a calibration argument and though it was not needed and so did not appear in \cite{ABCDS}, we note here
 that even in the limiting case where $\beta=0$ so that $V_\beta(p)=V_0(p)=\Theta(p)$, the calibration approach still works. The result is an
 identification of certain ellipses centered at the origin as closed isoperimetric curves for the homogeneous metric addressed in this section.
 \begin{cor}\label{ellipse} For any $p_0\in\R^2\setminus\{0\}$ the ellipse $\tilde{r}(p)=\tilde{r}(p_0)$ minimizes $E_H$ among all competing closed curves passing through $p_0$ and enclosing the same amount of Euclidean area $A$. Furthermore, the value of $E_H$ for this minimizing ellipse is given by $(\lm_1+\lm_2)A$.
  \end{cor}
 \begin{proof}
 The calibration argument for the Corollary is as follows. Define the $1$-form \[\omega:=-\frac{\lm_2}{\lm_1+\lm_2}p_2dp_1+\frac{\lm_1}{\lm_1+\lm_2}p_1dp_2,\]
 and note that $d\omega$ is just the Euclidean area form $dp_1\,dp_2$. 
 Parametrize the ellipse by degenerate arclength, say $\overline{\gamma}:[0,L_{\overline{\gamma}}]\to\R^2$, by expressing it as the integral curve of the vector field $V_0$ so that
 \[
\overline{\gamma}\,'=V_0(\overline{\gamma}),\quad\overline{\gamma}(0)=p_0\quad\mbox{where}\quad V_0(p_1,p_2):=\frac{\big(-\lm_2p_2,\lm_1p_1\big)}{F_H(p)^2}
 \]
 and we have that $F(\overline{\gamma}(\ell))\abs{\overline{\gamma}'}(\ell)=1.$ Then let $\gamma:[0,L_{\gamma}]\to\R^2$ be any other closed curve parametrized by degenerate arclength with $\gamma(0)=p_0$ that satisfies the same area constraint. Then by Stokes Theorem we have
 $\mathcal{A}(\gamma):=\int_{\gamma}p_1\,dp_2= \int_{\gamma}\omega=\int_{\overline{\gamma}}\omega.$
 Finally we compute that
 \[
 (\lm_1+\lm_2)\int_{\overline{\gamma}}\omega=\int_0^{L_{\overline{\gamma}}}\big(\lm_2\overline{\gamma}^{(2)},\lm_1\overline{\gamma}^{(1)}\big)\cdot 
 \overline{\gamma}'\,d\ell=\int_0^{L_{\overline{\gamma}}}1\,d\ell=L_{\overline{\gamma}}=E(\overline{\gamma})
 \]
 while
 \[
 (\lm_1+\lm_2)\int_\gamma\omega< \int_0^{L_{\gamma}}\abs{\big(\lm_2\gamma^{(2)},\lm_1\gamma^{(1)}\big)}\cdot 
\abs{\gamma'}\,d\ell=\int_0^{L_{\gamma}}1\,d\ell=L_{\gamma}=E(\gamma).
 \]
 Since $\int_{\overline{\gamma}}\omega$ is the amount of Euclidean area $A$ enclosed by the ellipse $\overline{\gamma}$ we are done.
 \end{proof}
 \begin{cor}\label{FHvert} For any $A\in\R$  one has the identity
 \[
 d_{F_H}\big((0,0,0),(0,0,A)\big)=\big(\lambda_1+\lambda_2\big)\abs{A}.
 \]
 \end{cor}
 \begin{proof}
 For any $\e>0$ we note from \eqref{tilder} that $\tilde{r}\big(\e,0)=\frac{1}{2}\lambda_1\e^2$ and so the solution to \eqref{quadprob} with the choice $p_0=(\e,0)$
 is the curve $\gb$ with the angle $\beta$ given by
 \[
 \frac{1}{2}\left(\frac{\lambda_1\e^2}{\lambda_1+\lambda_2}\right)\cot\beta=A+O(\e),
 \]
 cf. \eqref{cot}. Hence, by \eqref{lengthiden} we see that
 \[
 d_{F_H}\big((0,0,0),(\e,0,A)\big)=\sqrt{(A+O(\e))^2(\lambda_1+\lambda_2)^2+\frac{1}{4}\lambda_1^2\e^4}.
 \]
 Sending $\e\to 0$ and appealing to the continuity of the metric $d_{F_H}$ we arrive at the result.
 \end{proof}
 
 \section{Examples of non-existence}
 
 In light of Proposition \ref{goodprojection}, we see that a solution to \eqref{mainprob} will exist if and only if a 
 three-dimensional geodesic guaranteed by Theorem \ref{Gromov} contains no
 vertical segments over either of the wells. Often the vertical fiber over a well is not a three-dimensional geodesic and therefore none of its segments can be part of any geodesic. In this case Proposition \ref{goodprojection} applies and the projection of the three-dimensional geodesic provides a solution to
 \eqref{mainprob} for all values of the constraint. We give such an example at the end of the next sub-section. On the other hand, even if {\em both} vertical fibers over the wells {\em are} three-dimensional geodesics, 
 solutions to \eqref{mainprob} may still exist for all $A$-values. Such is the case of the homogeneous metrics considered in \cite{ABCDS} and reviewed in Section 3.
 
 In this section we present examples of explicit metrics and concrete areas $A$ for which the problem \eqref{mainprob} has no solution.

 \begin{center}
{\bf Non-existence with one potential well}
\end{center}

 We begin with an example of non-existence for the case of $F$ vanishing at only one point, which we take to be the origin in $\R^2$. We will then
 use this non-existence example to build an example of non-existence for the two well problem \eqref{mainprob}. For our one well example, we will
 take $F^2$ to be a radial function $F^2=F^2(r)$, with $r=\abs{p}$.  Specifically, we consider the case
 \beq
 F^2(r)=r^2+br^4\;\mbox{for}\;b>0.\label{Wexp}
 \eeq
 Then for this $F^2$ and for any $p_0\in\R^2\setminus\{(0,0)\}$ and $A\in\R$ we consider the problem
 \beq
 \inf\int_a^b\sqrt{\abs{\gamma}^2+b\abs{\gamma}^4}\abs{\gamma'}\,dt,\label{onewell}
 \eeq
where the infimum is taken over all locally Lipschitz continuous curves $\gamma:[a,b]\to\R^2$  such that $\gamma(a)=p_0,\;\gamma(b)=(0,0)$
and such that $\tilde{\mathcal{A}}(\gamma):=\int_{\gamma}\omega_1=\tilde{A}$  where 
\[\omega_1:=\frac{1}{2}r^2\,d\theta=-\frac{1}{2}\big(-p_2dp_1+p_1dp_2\big)\]
 and $\tilde{A}$ is any
given real number.

Here we have changed the constraint from the one used earlier in this paper, namely $\mathcal{A}(\gamma):=\int_{\gamma}\omega_0=A$ where $\omega_0:=p_1dp_2$, to $\tilde{\mathcal{A}}(\gamma)=\tilde{A}$ for convenience in working with polar coordinates, but we point out that since $d\omega_1=d\omega_0$, the two constraints are equivalent in the following sense:  If we let
$\ell_0$ denote the directed line segment from the origin to the point $p_0$ then by Stokes Theorem for any curve $\gamma$ joining $p_0$ to the origin we have
\beq
\tilde{\mathcal{A}}(\gamma)-\mathcal{A}(\gamma)=\int_{\gamma} \omega_1-\omega_0=\int_{\ell_0} \omega_1-\omega_0=-\int_{\ell_0}\omega_0=:C_0.\label{nodiff}
\eeq
Consequently, specifying a constraint value $\mathcal{A}(\gamma)=A$ is the same as specifying
the constraint value $\tilde{\mathcal{A}}(\gamma)=A+C_0$.

Now we will argue that problem \eqref{onewell} has a solution if and only if 
\beq
\abs{\tilde{A}}\leq\frac{\sqrt{\abs{p_0}}}{2\sqrt{b}}.\label{nomore}
\eeq
To this end, given any admissible curve $\gamma=\gamma(t)$ mapping say $[0,1]$ into $\R^2$  we introduce two new scalar dependent variables $R(t)$ and $\alpha(t)$ via
\beq
R(t):=\abs{\gamma(t)}^2\quad\mbox{and}\quad \alpha(t):=4\int_{\gamma([0,t])}\omega_1.\label{newvbles}
\eeq
Then a routine calculation reveals the relation
\[
\int_0^1\sqrt{\abs{\gamma}^2+b\abs{\gamma}^4}\abs{\gamma'}\,dt
=\int_0^1\frac{ \sqrt{1+bR}}{2}\sqrt{(R')^2+(\alpha')^2}\,dt=:\tilde{E}(R,\alpha).
\]
Thus, this change of variables ``desingularizes" the problem in favor of one which is conformal to the standard Euclidean metric in the right-half $R-\alpha$ plane with a strictly positive conformal factor. What is more, the constraint now simply becomes part of a Dirichlet condition,
\beq
\big(R(0),\alpha(0)\big)=\big(\abs{p_0},0\big)\quad\mbox{and}\quad \big(R(1),\alpha(1)\big)=\big(0,4\tilde{A}\big).\label{newDir}
\eeq
Of course there is also the crucial constraint $R\geq 0$ and the equivalence of the constrained problem \eqref{onewell} and the minimization of
$\tilde{E}$ subject to \eqref{newDir} relies on the solution to the latter problem avoiding any line segment along the $\alpha$-axis where $\omega_1$ is not well-defined, thereby making the variable $\alpha$ no longer an accurate measure of area..
We should remark that this idea of introducing a variable to keep track of the constraint (area) value was already employed in our earlier reformulation of \eqref{mainprob} in terms of three-dimensional geodesics, but the difference here is that due to the assumed radial dependence of the conformal factor, the resultant problem now is still two-dimensional.

The minimization of $\tilde{E}(R,\alpha)$ subject to the Dirichlet conditions \eqref{newDir}  and $R\geq 0$ can be carried out explicitly. First, we observe that since the conformal factor is independent of $\alpha$ and since it is a monotone increasing function of $R$, it is easy to argue that whenever $R$ is positive, it is always optimal to consider competitors that are graphs over the positive $R$ axis. In other words, it is always unnecessarily costly for a curve to double back when progressing from $R=\abs{p_0}$ to $R=0$. What could happen, however, is that the minimizer 
includes a vertical line segment along the $\alpha$-axis, that is along the line $R=0$. This phenomenon is exactly the scenario where the three-dimensional minimizing curve of Theorem \ref{Gromov} does {\em not} equate to a solution to \eqref{mainprob}.

With this in mind, we now ask: What is the set of $\tilde{A}$ values such that the minimizer to $\tilde{E}(R,\alpha)$ is never vertical over the $\alpha$-axis and so remains a graph for all $t\in [0,1]$?
Phrasing the minimization then for graphs $\alpha=f(R)$, we arrive at the problem
\beq
\inf_f\int_0^{\abs{p_0}}\frac{ \sqrt{1+bR}}{2}\sqrt{1+(f'(R))^2}\,dR\label{graph}
\eeq
subject to the boundary conditions $f(0)=4\tilde{A}$ and $f(\abs{p_0})=0$.

Criticality for this problem takes the form
\[
\sqrt{1+bR}\frac{f'}{\sqrt{1+(f')^2}}=C_1
\]
for a constant $C_1$ and by considering the limit $R\to 0$, we see that evidently $\abs{C_1}\leq 1$. Integrating this ODE we obtain
a family of parabolas $\{f_{C_1}\}$ opening to the right given by
\beq
f_{C_1}(R)=-\frac{2C_1}{b}\sqrt{1-C_1^2+bR}+D\label{parabolas}
\eeq
with the minus sign arising if we consider, for example, the case $\tilde{A}>0$ and $D$ is chosen so that $f_{C_1}(\abs{p_0})=0$. 
The question we posed above \eqref{graph} now takes the concrete form: What is the largest value of $\tilde{A}$ such that one of these lower branches of parabolas makes the transition from $4\tilde{A}$ down to zero
on the $R$ interval $[0,\abs{p_0}]$? After a little bit of algebra we arrive at the answer, namely \eqref{nomore}, which corresponds to the parabola $f_{1}$ with $C_1=1$ whose vertex meets the $\alpha$-axis tangentially. 

For any $\abs{\tilde{A}}>\frac{\sqrt{\abs{p_0}}}{2\sqrt{b}}$, the parabola satisfying \eqref{newDir} will bow into the left half-plane where $R$ is negative. Such a curve is indeed length minimizing in the metric defined by $\tilde{E}$
provided it stays in the half-plane $R>-\frac{1}{b}$, but it is inadmissible given the requirement $R\geq 0$ forced by \eqref{newvbles}. Since the conformal factor is always cheaper along the $\alpha$-axis than in the region $R>0$, the optimal resolution for such large $\tilde{A}$ is clearly to join two
points along the $\alpha$-axis with a vertical segment. (See Figure 1a.) Going back to the formulation \eqref{ddefn},  this provides an example of a three-dimensional minimizer from Theorem \ref{Gromov} whose projection fails to provide a solution to problem \eqref{onewell}; hence  \eqref{onewell} can have no solution. 
\begin{figure}[h]
{\includegraphics[trim={80 380 0 120},clip=true]{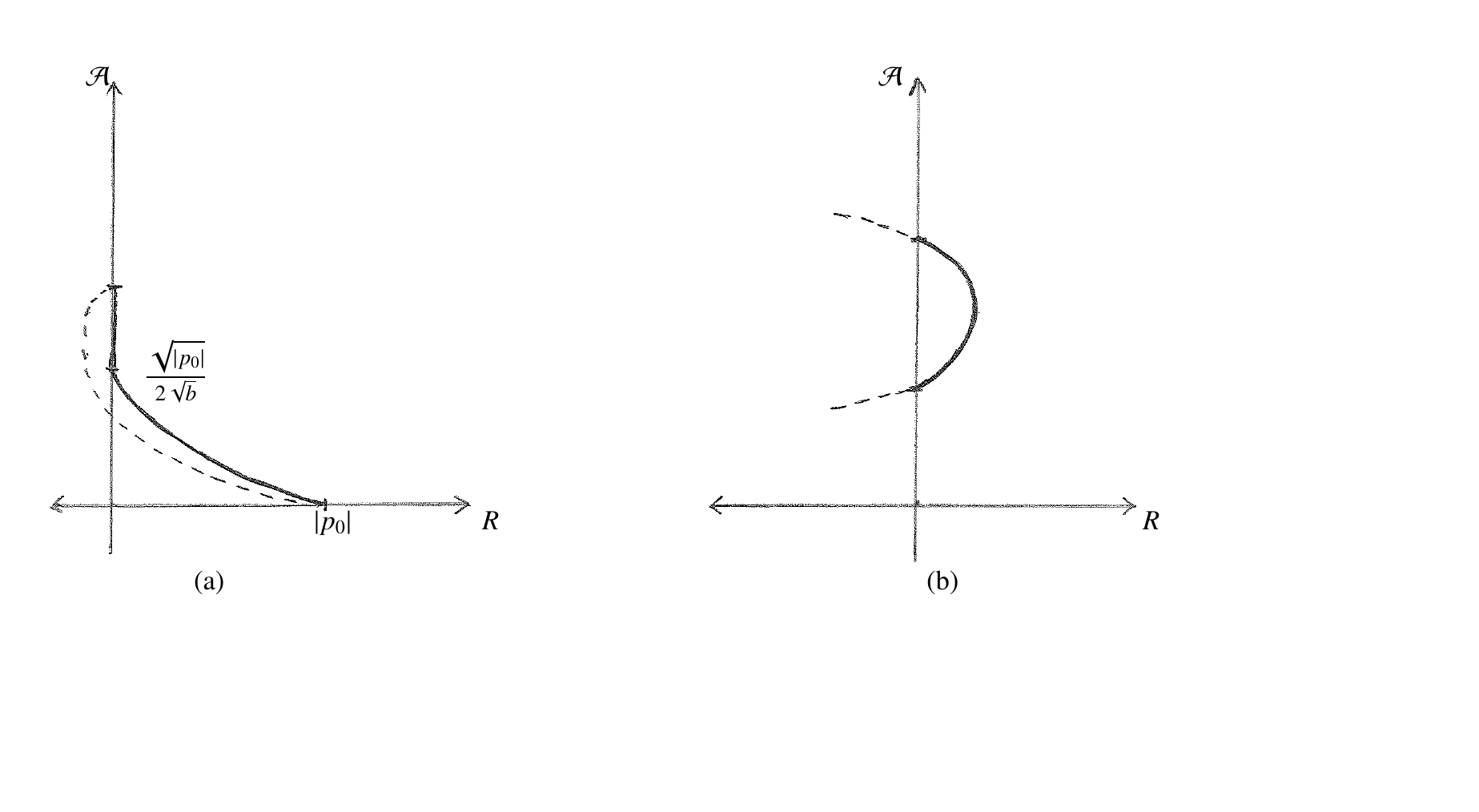}}
\caption{(a) The appearance of a vertical segment once the value of area exceeds a critical value for $b>0$. The dashed parabola depicts the optimal curve if
negative $R$ values were allowable. (b) When $b<0$ no vertical segment is geodesic since the parabola is shorter.}
\end{figure}

We should remark that while for constraint values satisfying \eqref{nomore}, the minimizer in the $R\alpha$ plane is a piece of a parabola, when one reverts to standard polar coordinates, one finds
the optimal curve in the $p_1p_2$ plane is given by an (infinite) spiral satisfying the relation
\beq
\frac{d\theta}{dr}=-\frac{C_1}{r\sqrt{1-C_1^2+br^2}}\quad\mbox{for}\;0<r\leq \abs{p_0}.\label{thetar}
\eeq
Working with \eqref{thetar}, one readily checks that these minimizing spirals satisfy the Euler-Lagrange equation \eqref{Euler} with Lagrange mulitplier $\lm=2C_1$ so that by our bound on $C_1$ we see that 
\beq
\abs{\lm}\leq 2.\label{lesstwo}
\eeq 
This inequality will be used in the two-well example below.

A more interesting consequence of this formula is that when applied to the limiting case $C_1=1$,
the derivative $\frac{d\theta}{dr}$ behaves like $Const./r^2$ thus showing that this $E$-minimizing spiral has {\em infinite} Euclidean arclength, in contrast for instance to the situation in the homogeneous metric, cf. \eqref{euclidbd}.

Before moving on to an example of non-existence for the two-well problem, let us consider the situation where $F^2$ is given by \eqref{Wexp}
in a neighborhood of the origin but with the coefficient $b$ now taken to be {\em negative}. Then \eqref{parabolas} reveals critical points of 
$\tilde{E}$ consisting of parabolas opening to the left.  We consider now two nearby points on the $\alpha$-axis. Being close together one knows there is a unique geodesic joining them in this metric. We conclude that the unique geodesic joining these two points is the left-opening parabola. In particular, it must be shorter than the vertical line segment joining them. See Figure 1b. 

Phrased in terms of the metric $d_F$ this means the vertical fiber over the origin is not a three-dimensional geodesic since that would force arbitrarily small vertical segments to be length-minimizing. Applying this to the two-well setting, we conclude from Proposition \ref{goodprojection} that if $F$ is locally described by \eqref{Wexp} near both of its wells with $b<0$, then a solution exists to \eqref{mainprob} for all values of the constraint $A$.

 \begin{center}
{\bf Non-existence with two potential wells}
\end{center}

We will now use the previous example to construct a potential $F^2$ having two potential wells for which problem \eqref{mainprob} has no solution.
For a constant $k>1$, later taken to be sufficiently large, we define the function
\beq
g(r):=\left\{\begin{matrix} r^2+\big(k^2-1\big)r^4&\mbox{for}\;r\leq 1,\\
k^2&\mbox{for}\;r\geq 1.\end{matrix}\right.
\eeq
Then we take $F^2:\R^2\to\R$ to be a given in the left half-plane by $g$ in a polar coordinate system centered at $(-1,0)$ and in the right half-plane we take the even reflection of this function. The resulting function $F^2$ vanishes at the points $(\pm1,0)$, is locally radial about these wells in discs which we denote by $D_+$ and $D_-$ of radius one centered at $(\pm1,0)$ that touch at $(0,0)$. Outside these discs, $F$ is the constant $k^2$.  While this $F$ is only Lipschitz continuous across the boundary of the discs, this amount of smoothness will suffice for our purposes, though the example could be readily modified to provide $C^2$ examples. We also observe that the unique unconstrained planar geodesic here is clearly the horizontal line segment joining the two wells so that in particular by \eqref{nodiff} we have $\tilde{\mathcal{A}}(\gamma)=\mathcal{A}(\gamma)$ for any curve joining the wells.

We will now argue that there is an interval of area values $\mathcal{A}(\gamma)=A$--actually two intervals if one allows for $A$ both negative and positive--for which there is no solution to \eqref{mainprob}.  First we note that the line segment joining the two wells is clearly the minimizing two-dimensional geodesic in the sense of \eqref{geo} for this conformal factor $F$. Now we suppose $\gamma$ is a solution to \eqref{mainprob} for a fixed value of $A$ which we may as well take to be positive, since the case of negative $A$ is entirely the same.

We observe that within either $D_+$ or $D_-$, $\gamma$ solves a one-well problem \eqref{onewell} as described in the previous example with
the value of $b$ given by $k^2-1$, some (unknown) value of $A$ and the (unknown) point $p_0$ satisfying $\abs{p_0}=1$. Since there is a global Lagrange multiplier associated with this problem, we see that $\gamma$ must solve
\eqref{Euler} in both discs and outside the two discs for the same value of $\lambda$. Hence within both discs, the solution must be a spiral given by
\eqref{thetar} with the {\em same} constant $C_1=\lm/2$ and as before $\abs{\lm}\leq 2$. Consequently, we conclude that within both discs the two spirals must agree up to a rotation. What is more, the amount of area $\mathcal{A}(\gamma)$ must agree and so by \eqref{nomore} the total cannot exceed the value $\frac{1}{\sqrt{k^2-1}}$, which is small when $k$ is large.

Now outside of the two discs, since the metric is just the constant $k$ multiplying the standard Euclidean metric, we know that $\gamma$ is a circular arc. Since $\gamma$ is $C^1$ across the boundary of the discs and since the spirals in the two discs agree up to a rotation, we conclude that the circular 
arc makes the same angle with the two circles forming the boundaries of these discs. Furthermore, from \eqref{thetar}, this contact is nearly orthogonal when
$k$ is taken to be large.
Then denoting the radius of curvature of this circular arc by say $\overline{R}$, one can easily check that the condition of constancy of the Lagrange multiplier within and without the discs implies through \eqref{lesstwo} that
\[
\frac{2k}{\overline{R}}=\lm\leq 2,
\]
so that the curvature of this arc is large when $k$ is taken to be large.

Elementary geometric considerations then lead to the conclusion that in fact $\gamma$, and in particular the circular arc, must be symmetric about the vertical $p_2$ axis. For $k$ large there are only two possibilities: either the circular arc is close to the $p_1$ axis, thereby contributing a small amount to the total value of $\mathcal{A}(\gamma)$ or else it spans a huge amount of area in region $p_2>0$, see Figure 2.

\begin{figure}[h]
\centering
\includegraphics[trim=3cm 3cm 3cm 4cm, clip=true, width=10cm,height=10cm,keepaspectratio, angle=90]{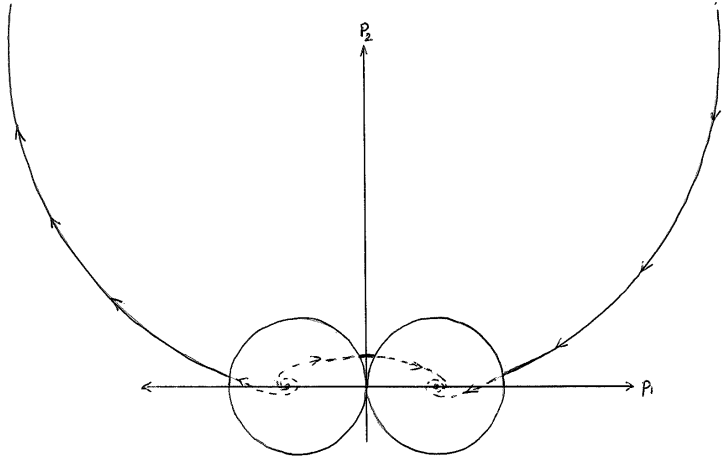}
\caption{Counter-example for intermediate area values: When $k$ is large, either a tiny circular arc bridges the two spirals or else a huge circular arc is required.}
\end{figure}


We conclude that the only attainable values of area $\mathcal{A}$ for this example are either 
\[(-\frac{1}{\sqrt{k^2-1}}-\delta,\frac{1}{\sqrt{k^2-1}}+\delta)\quad\mbox{for}\quad \delta\ll 1\quad\mbox{or}\quad\abs{A}\geq M\quad\mbox{for some}\quad M\gg 1,\]
 so there is a large interval of positive $A$-values for which no solution exists.

Finally we remark one could likely produce a different type of example of non-existence by taking $F$ in neighborhoods of the wells to again resemble our one-well non-existence example but then taking $F$ to be huge outside these neighborhoods except for a narrow `trench' along the
line segment joining the wells. Such an example would presumably lead to non-existence for all constraint values exceeding some positive constant in absolute value.

\section{Existence for constraint values near $\mathcal{A}(\gamma_0)$}

In Section 2 we introduced a three-dimensional re-phrasing of our problem \eqref{mainprob} and in Proposition \ref{goodprojection} we noted that the three-dimensional geodesic has a projection solving \eqref{mainprob} provided this geodesic does not contain any line segment over either of the wells $\pmi$ or $\ppl$. In the previous section, however, we exhibited an example to show that such a vertical line segment can at times exist, leading to non-existence for \eqref{mainprob}. In this section we turn to our main result asserting that for constraint values near that of the 
unconstrained planar geodesic $\gamma_0$ there is never a vertical line segment for the length-minimizing geodesic in three-dimensions. Hence, its projection in the plane will solve our problem.

Before getting into the details of the proof we want to describe the main idea behind the argument. For heuristic purposes, let us suppose that for an interval of area constraint values $A$ we have a family of minimizers to our problem \eqref{mainprob} smoothly depending on $A$, which we denote by $\gamma_A$. For each $A$ then the curve $\gamma_A$ will be a critical point satisfying the ODE \eqref{Euler} with corresponding Lagrange multiplier $\lm_A=F^2(\gamma_A)\kappa_g(\gamma_A)$, cf. Remark \ref{gc}. Computing the derivative of length $E(\gamma_A)$ with respect to area $A$ we find
\begin{eqnarray*}
&&\frac{d\,E(\gamma_A)}{d\,A}=\frac{d}{dA}\int F(\gamma_A)\abs{\gamma_A'}=\int \bigg[-\left(F(\gamma_A)\,\frac{\gamma_A'}{\abs{\gamma_A'}}\right)'+\abs{\gamma_A'}\nabla F(\gamma_A)\bigg]
\cdot\frac{\partial\gamma_A}{\partial A}\\
&&=-\lm_A\int (\gamma_A')^\perp\cdot \frac{\partial\gamma_A}{\partial A}.
\end{eqnarray*}
Then differentiating the constraint condition $\mathcal{A}(\gamma_A)=A$ with respect to $A$ as well we see that 
\beq
-\int (\gamma_A')^\perp\cdot \frac{\partial\gamma_A}{\partial A}=1
\quad
\mbox{so that}\quad
\frac{d\,E(\gamma_A)}{d\,A}=\lm_A.\label{dEdA}
\eeq
Now for the constraint value associated with the unconstrained planar geodesic $\gamma_0$, namely $A=\mathcal{A}(\gamma_0)$, we know $\lm=0=\kappa_g$ so we should expect that for nearby $A$-values, $\lm_A$ should be {\em small}. On the other hand, when a minimizing curve $\gamma_A$ approaches a zero of $F$ where $F$ is well-approximated by its quadratic Taylor polynomial $F_H$, we might anticipate that the rate of change of length with respect to area--that is, $\lm_A$-- should resemble more and more that of the quadratic case for which we have a complete solution as described in Section 3. In particular, through Corollary \ref{FHvert} we know that the cost of a vertical line segment measured in length per area for the homogeneous metric $d_{F_H}$ is $\lm=\lm_1+\lm_2$, which is `far' from zero. Therefore when $A$ is near $\mathcal{A}(\gamma_0)$,
the three-dimensional geodesic cannot become vertical over either well.

We begin with a key lemma that says this rate of change for vertical segments measured with respect to our general metric $d_F$ is given by
the value inherited from $d_{F_H}$.

\blemma\label{vertderiv} 
Translating either $\ppl$ or $\pmi$ to the origin, the value of the rate of change of length with respect to area for a vertical segment over either well is given by
\beq
\lim_{\alpha\to 0}\frac{d_F\big((0,0,0),(0,0,\alpha)\big)}{|\alpha|}=\lambda_1+\lambda_2,\label{alphaclaim}
\eeq
where $F$ takes the form  \beq
F(p_1,p_2)=\sqrt{\lm_1^2 p_1^2+\lm_2^2 p_2^2+O(|p|^3)}\quad\mbox{in a neighborhood of the origin}\label{taylor}
\eeq
for positive constants $\lm_1$ and $\lm_2$.

\elemma
\brk
The pair of positive constants $\lm_1$ and $\lm_2$ for the Taylor development of $F$ near $\ppl$ is not required to be the same pair for $F$ near
 $\pmi$.
\erk
\begin{proof}

We will show this for $\alpha\to 0^+$ with the result for $\alpha\to 0^-$ following similarly.
To see this, we invoke the triangle inequality
\[
\frac{d_F\big((0,0,0),(0,0,\alpha)\big)}{\alpha}\leq \frac{d_F\big((0,0,0),(\alpha^2,0,\alpha)\big)}{\alpha}+
\frac{d_F\big((\alpha^2,0,\alpha),(0,0,\alpha)\big)}{\alpha}
\]
and for the last term we have
\[
\frac{d_F\big((\alpha^2,0,\alpha),(0,0,\alpha)\big)}{\alpha}\leq
\frac{\int_0^{\alpha^2} F(t,0)\,dt}{\alpha}\to 0.
\]
Now let $\gamma_{\beta(\alpha)}$ denote the projection of the optimal three-dimensional curve in the homogeneous metric $d_{F_H}$, that is, the planar curve joining $(0,0)$ to $(\alpha^2,0)$ and satisfying the constraint value $\mathcal{A}=\alpha$ as
described in Theorem \ref{quadthm}. In particular, $\beta(\alpha)$ is given by \eqref{cot} with $p_0=(\alpha^2,0)$ and $A=\alpha$.
Since $\abs{\gamma_{\beta(\alpha)}}\leq \alpha^2$, we are working near the origin so we can use the estimate
\[
\abs{F(p)-F_H(p)}\leq C\abs{p}^2,
\]
cf. \eqref{FH} and apply \eqref{euclidbd} to argue that
\begin{eqnarray*}
&&
\limsup_{\alpha\to 0}\frac{d_F\big((0,0,0),(0,0,\alpha)\big)}{\alpha}\leq \limsup_{\alpha\to 0}\frac{d_F\big((0,0,0),(\alpha^2,0,\alpha)\big)}{\alpha}
\leq\limsup_{\alpha\to 0}
\frac{\int F(\gamma_{\beta(\alpha)})|\gamma_{\beta(\alpha)}'|}{\alpha}\\
&&
\leq \limsup_{\alpha\to 0}\left\{
\frac{\int F_H(\gamma_{\beta(\alpha)})|\gamma_{\beta(\alpha)}'|}{\alpha}+\frac{1}{\alpha}\int\abs{F(\gamma_{\beta(\alpha)})
-F_H(\gamma_{\beta(\alpha)})}|\gamma_{\beta(\alpha)}'|\right\}\\
&&
\leq \limsup_{\alpha\to 0}\left\{ \frac{d_{F_H}\big((0,0,0),(\alpha^2,0,\alpha)\big)}{\alpha}+\frac{1}{\alpha}C\alpha^2
\int |\gamma_{\beta(\alpha)}'|\right\}\\
&&=\lim_{\alpha\to 0} \frac{d_{F_H}\big((0,0,0),(\alpha^2,0,\alpha)\big)}{\alpha}.
\end{eqnarray*}
As \eqref{lengthiden} gives us an explicit formula for $d_{F_H}\big((0,0,0),(\alpha^2,0,\alpha)\big)$, we conclude that
\beq
\limsup_{\alpha\to 0}\frac{d_F\big((0,0,0),(0,0,\alpha)\big)}{\alpha}\leq
\lim_{\alpha\to 0}\frac{\sqrt{\alpha^2\big(\lambda_1+\lambda_2\big)^2+\frac{1}{4}\lambda_1^2\alpha^4}}{\alpha}=\lambda_1+\lambda_2.
\label{upperbd}
\eeq

To complete the demonstration of \eqref{alphaclaim} we must show the reverse inequality. For this purpose we introduce the scaled conformal factor 
$
F_{\theta}(p):=\frac{1}{\theta}F(\theta\,p)
$ for any $\theta>0$ and the corresponding functional 
\[
E_{\theta}(\zeta):=\int F_{\theta}(\zeta)\abs{\zeta'}
\]
for any planar curve $\zeta$. We note that $F_{\theta}\to F_H$ as $\theta\to 0$ uniformly on compact subsets of $\R^2$ so that
$d_{F_{\theta}}(P,Q)\to d_{F_H}(P,Q)$ for all $P$ and $Q$ in $\R^3$ as well. Here $d_{F_{\theta}}$ denotes the metric as defined in \eqref{ddefn} but with the conformal factor taken to be $F_{\theta}.$
Now for any $\alpha$ let $\gamma_{\alpha}:[a,b]\to\R^2$ be any curve with endpoints $\gamma_{\alpha}(a)=
\gamma_{\alpha}(b)=(0,0)$ such that $\mathcal{A}(\gamma_{\alpha})=\alpha$ and such that
\[
E(\gamma_{\alpha})\leq d_F\big((0,0,0),(0,0,\alpha)\big)+\alpha^2;
\]
that is, $\gamma_{\alpha}$ is a good competitor in the minimization of $E$ subject to these endpoint and constraint values.
We then observe that
\begin{eqnarray*}
d_{F_{\sqrt{\alpha}}}\big((0,0,0),(0,0,1)\big)&&\leq E_{\sqrt{\alpha}}\big(\frac{1}{\sqrt{\alpha}}\gamma_{\alpha}\big)\quad\bigg(\mbox{since}\;
\mathcal{A}\left(\frac{1}{\sqrt{\alpha}}\gamma_{\alpha}\right)=1\bigg)\\
&&=\frac{1}{\alpha}E(\gamma_{\alpha})\leq \frac{d_F\big((0,0,0),(0,0,\alpha)\big)}{\alpha}+\alpha.
\end{eqnarray*}
Letting $\alpha\to 0$ and invoking Corollary \ref{FHvert} we arrive at the desired reverse inequality
\[
\liminf_{\alpha\to 0}\frac{d_F\big((0,0,0),(0,0,\alpha)\big)}{\alpha}\geq d_{F_H}\big((0,0,0),(0,0,1)\big)= \lambda_1+\lambda_2
\]
and so \eqref{alphaclaim} is established.
\end{proof}

Now we can establish our main result.

\bthm\label{mainresult} Let $\gamma_0:[0,1]\to\R^2$ be an unconstrained minimizing geodesic joining $\pmi$ to $\ppl$, i.e. a solution
to the variational problem \eqref{geo} with $p=\pmi$ and $q=\ppl.$ Then there
exists a number $\e_0>0$ such that for all $A$ in the interval $\big(\mathcal{A}(\gamma_0)-\e_0,\mathcal{A}(\gamma_0)+\e_0\big)$, there exists a solution to \eqref{mainprob}. \ethm
\begin{proof}

For any $\e\in\R$, later to be taken small in absolute value, let $\Gamma_\e:[0,1]\to\R^3$ be the length-minimizing geodesic guaranteed by Theorem \ref{Gromov} satisfying $L(\Gamma_\e)=d(P_\e,Q_\e)$ where $P_\e$ and $Q_\e$ are any two points in $\R^3$ such that  
\beq
\Pi(P_\e)=\pmi,\quad \Pi(Q_\e)=\ppl\quad \mbox{and}\quad 
Q_\e^{(3)}-P_\e^{(3)}=A_0+\e.\label{slide}
\eeq
Throughout the proof we will denote by $\gamma_\e$ the projection $\Pi(\Gamma_\e)$.

In light of Theorem \ref{goodprojection} it will suffice to demonstrate that $\Gamma_\e$ does not possess any vertical line segment
over either $\pmi$ or $\ppl$.
To this end, we will pursue an argument by contradiction and suppose that this condition {\it fails} for a sequence $\e_j\to 0$. For it to fail near $\pmi$
means that we assume:
\beq
\mbox{For a sequence}\;\e_j\to 0\;\mbox{there is a nontrivial interval}\;I_j\subset [0,1]\;\mbox{such that}\; \gamma_{\e_j}(t)=\pmi\;\mbox{for all}\;t\in I_j.\label{segment}
\eeq
 The argument precluding such a scenario from happening near $\ppl$ is identical.

After an appropriate translation of coordinates in the $p_1p_2$-plane we may take $\pmi=(0,0)$ and after an appropriate rotation of coordinates so that the coordinate axes are aligned with the eigenvectors of the Hessian $D^2W(\pmi)$ we may assume
the function $F=\sqrt{W}$ near $\pmi$ takes the form given by \eqref{taylor}.

 Our strategy will be to first construct a family of 
competitors smoothly depending on a parameter, say $\alpha$, that involves either enlarging or shrinking the length of the assumed line segment within $\Gamma_{\e_j}$ by an amount $\alpha$ while compensating for this change by subtracting or adding a ``bump" to $\Gamma_{\e_j}$ away from the segment so as to maintain the endpoint condition \eqref{slide}. Since for this portion of the argument, the value of $\e_j$ is not varying, for ease of notation we will temporarily suppress this dependence and write simply $\Gamma$ for $\Gamma_{\e_j}$. Once we need to consider $\Gamma_{\e_j}$ for $\e_j$ small we will revert to the original notation.

For any small $\alpha\in\R$ let us describe more precisely this smooth deformation of $\Gamma$, which we denote by $\Gamma^{\alpha}$, as follows. We alter
the vertical line segment $\Gamma(I_j)$ by length $\alpha$. Recall that the third component of $\Gamma$ effectively keeps track of the value of the area constraint $\mathcal{A}$. Therefore, in order to create a family of deformations that preserves the total change in the third component, we work locally near some planar sub-arc of $\gamma:=\Pi(\Gamma)$ away from $\pmi$ and $\ppl$,
say $\gamma_{|_{[a,b]}}$ for some parameter interval $[a,b]$ and simultaneously modify this sub-arc by augmenting it with a bump that subtracts $\alpha$ from the constraint value $\mathcal{A}(\gamma).$ For example, we may suppose that on the interval $[a,b]$, this sub-arc is expressible as a graph $y=f(x)$ of a smooth scalar function $f$. Then
letting $g:[a,b]\to\R$ be any smooth function satisfying the conditions 
\[
g(a)=0=g(b)\quad\mbox{and}\quad \int_a^bg(x)\,dx=1,
\]
we consider the family of perturbations $\tilde{\gamma}_{\alpha}:[a,b]\to\R^2$ of $\gamma_{|_{[a,b]}}$ given by $x\mapsto \big(x,f(x)+\alpha g(x)\big)$. Note that
\[
\mathcal{A}\big(\tilde{\gamma}_{\alpha}\big):=\int_a^b \tilde{\gamma}_{\alpha}^{(1)}\tilde{\gamma}_{\alpha}^{(2)}\,' \,dx=
\int_a^b x\big(f'(x)+\alpha g'(x)\big)\,dx=\mathcal{A}\big(\gamma_{|_{[a,b]}}\big)
+\alpha\int_a^b xg'(x)\,dx\\
=\mathcal{A}\big(\gamma_{|_{[a,b]}}\big)-\alpha.
\]
Along this parameter interval we then define the third component of $\Gamma^\alpha$ through requirement \eqref{assoc}.

As $\Gamma^0=\Gamma$ the family $\{\Gamma^\alpha\}$ constitutes a smooth deformation of the minimizing geodesic $\Gamma$ that preserves its endpoints
$P_\e$ and $Q_\e$ up to translation in the third coordinate direction, cf. \eqref{slide}. Invoking Lemma \ref{vertderiv} along with 
\eqref{Euler} of Proposition \ref{goodprojection} we can compute
\begin{eqnarray*}
\frac{d\,L(\Gamma^\alpha)}{d\alpha}_{|_{\alpha=0}}&&=\lim_{\alpha\to 0}\frac{d_F\big((0,0,0),(0,0,\alpha)\big)}{\alpha}
+
\frac{d}{d\alpha}_{|_{\alpha=0}}\left\{\int_a^b F\big(x,f(x)+\alpha g(x)\big)\sqrt{1+(f'+\alpha g')^2}\,dx\right\}\\
&&
=\lambda_1+\lambda_2+\int_a^b \left[-\left(F(\gamma)\frac{\gamma'}{\abs{\gamma'}}\right)'+\abs{\gamma'}\nabla F(\gamma)\right]
\cdot \big(0,g)\,dx\\
&&
=\lambda_1+\lambda_2-\lambda\int_a^b \big(\gamma'\big)^\perp\cdot (0,g)\,dx=\lambda_1+\lambda_2-\lambda\int_a^bg(x)\,dx=
\lambda_1+\lambda_2-\lambda,
\end{eqnarray*}
where $\lambda$ is the Lagrange multiplier arising in the Euler-Lagrange equation. Since $\Gamma$ is minimizing, necessarily the derivative above must vanish and so we conclude that a minimizer that includes a vertical segment must have Lagrange multiplier $\lambda$ given by
\beq
\lambda=\lambda_1+\lambda_2.\label{gotcha}
\eeq
We recall now that really $\Gamma=\Gamma_{\e_j}$, the length-minimizing geodesic subject to the $\e$-dependent conditions \eqref{slide}. Yet under the contradiction hypothesis \eqref{segment} that $\Gamma_{\e_j}$ includes a vertical line segment, the value of $\lambda$ is forced to satisfy condition \eqref{gotcha}
which is independent of $\e$.  We will now exploit this property to show that $\inf L(\Gamma_{\e_j})>E(\gamma_0)$ and easily reach a contradiction.

To this end, we recall the definition of the planar metric  $\overline{d}_F\big(p,q\big)$ given in \eqref{geo}
and we fix any positive $r<\frac{1}{2}\overline{d}_F(\pmi,\ppl)=\frac{1}{2}E(\gamma_0)$. Then we introduce the topological annulus in the plane 
\beq
\mathcal{D}_r:=\{p\in\R^2:\,r<\overline{d}_F(\pmi,p)<2r\}.\label{topannulus}
\eeq
We note that since each $\gamma_{\e_j}$ starts at $\pmi$ and ends at $\ppl$, it must contain at least one sub-arc lying inside $\mathcal{D}_r$ with endpoints $p_j$ and $q_j$ satisfying 
$\overline{d}_F(\pmi,p_j)=r$ and $\overline{d}_F(\pmi,q_j)=2r$ respectively. Denoting this sub-arc by $\zeta_j$ and parametrizing it so that
\beq
F\big(\zeta_j(\ell)\big)\abs{\zeta_j'(\ell)}=E(\zeta_j)\label{pam}
\eeq 
we obtain a family of curves $\{\zeta_j\}$ mapping $[0,1]$ to $\overline{\mathcal{D}_r}$. 
Also we observe that $r\leq E(\zeta_j)$ since the width in the metric $\overline{d}_F$ of the annulus is $r$, while
\[
 \sup_{j}E(\zeta_j)<\sup_j E(\gamma_{\e_j})\leq \sup_{\e\in [-1,1]}d_F\big((\pmi,0),(\ppl,\mathcal{A}(\gamma_0)+\e)\big)<\infty.
\]

We will now use these $j$-independent bounds as well as \eqref{gotcha} to establish the claim that
\beq
\liminf_{j\to\infty}E(\zeta_j)>r.\label{jclaim}
\eeq
We recall that under contradiction hypothesis \eqref{segment} we have shown that necessarily $\gamma_{\e_j}$ solves \eqref{Euler} with $\lambda$ given by \eqref{gotcha}. Then
we use the chosen parametrization \eqref{pam} to see that $\zeta_j$ solves the system of ODE's
\beq
-\frac{1}{E(\zeta_j)}\left(F^2(\zeta_j)\zeta_j'\right)'+\frac{E(\zeta_j)}{F(\zeta_j)}\nabla F(\zeta_j)=-(\lm_1+\lm_2)\left(\zeta_j'\right)^\perp.\label{newEuler}
\eeq
Since $F$ is $C^1$ and
\[
0<\min_{p\in \overline{D}_r} F(p)<\max_{p\in \overline{D}_r} F(p)<\infty,
\]
being a solution to this system implies a $C^2$ bound on the curves $\{\zeta_j\}$ that is independent of $j$. Hence, we obtain
a curve $\zeta_0$ that is the uniform $C^1$ limit of a subsequence of $\{\zeta_j\}$ and $\zeta_0$ must solve
\beq
-\frac{1}{E(\zeta_0)}\left(F^2(\zeta_0)\zeta_0'\right)'+\frac{E(\zeta_0)}{F(\zeta_0)}\nabla F(\zeta_0)=-(\lm_1+\lm_2)\left(\zeta_0'\right)^\perp
\eeq
(first weakly, but then by standard regularity theory, strongly). Then the claim \eqref{jclaim} follows because were it false, we would have $E(\zeta_0)=r$, making $\zeta_0$
an $E$-minimizing curve spanning $\mathcal{D}_r$ which would force $\zeta_0$ to solve the above ODE with zero right-hand side.

With \eqref{jclaim} in hand, we now have that for all $\e_j$ sufficiently small:
\[
L(\Gamma_{\e_j})-E(\gamma_0)\geq E(\gamma_{\e_j})-E(\gamma_0)>\frac{1}{2}\left(\liminf_{j\to\infty}E(\zeta_j)-r\right)=:a_0.
\]
Clearly by modifying $\gamma_0$, say with a small bump, we can accommodate the constraint requirement $\mathcal{A}=
\mathcal{A}(\gamma_0)+\e_j$ for $\e_j$ small to build a competing planar curve--and then three-dimensional curve via \eqref{assoc}-- whose extra length is less than $a_0$, thus contradicting the minimality
of $\Gamma_{\e_j}$ among curves joining $(\pmi,0)$ to $(\ppl,\mathcal{A}(\gamma_0)+\e_j)$ under assumption \eqref{segment}. Since no vertical line segment exists, we conclude from Proposition \ref{goodprojection} that $\gamma_\e$ solves \eqref{mainprob} provided $\abs{\e}$ is sufficiently small. 
\end{proof}
Next we establish that for values of $A$ near $\mathcal{A}(\gamma_0)$ the solutions from Theorem \ref{mainresult} contain no bubbles. This property will in particular be important when interpreting these curves as traveling waves to the Hamiltonian system discussed in Section 6.
\bthm\label{nobubbles} Denote by $\gamma_\e$ the solution to \eqref{mainprob} obtained in Theorem \ref{mainresult} for $A=\mathcal{A}(\gamma_0)+\e$ and $\abs{\e}<\e_0.$ Then there exists a positive value $\e_1\leq \e_0$ such that when $\abs{\e}<\e_1$ the curve $\gamma_\e$ only meets $\ppl$ and $\pmi$ at its endpoints.
\ethm
\begin{proof}
We again proceed by way of contradiction. Thus we assume that for some sequence $\e_j\to 0$:
\begin{eqnarray}
&&\mbox{There exist two sequences}\;0\leq s_{\e_j}< t_{\e_j}<1\;\mbox{such that}\nonumber\\
&&\gamma_{\e_j}(s_{\e_j})=\gamma_{\e_j}(t_{\e_j})=\pmi\;\mbox{while}\;\gamma_{\e_j}(t)\not\in\{\pmi,\ppl\}\;\mbox{for}\;s_{\e_j}<t<t_{\e_j}.
\label{loop}
\end{eqnarray}
We will reach a contradiction under the scenario \eqref{loop} in which, say,  $\ell_j:=\gamma_{\e_j}([s_j,t_j])$ constitutes a closed planar loop (i.e. a `bubble') passing through $\pmi$, by a line of reasoning similar to that employed in the existence proof of Theorem \ref{mainresult}.

 If we let $q_j\in \ell_j$ denote the furthest point on the loop from $\pmi$ in the $d_F$ metric, i.e.
\[
d_F(q_j,\pmi)=\max_{p\in\ell_j}d_F(p,\pmi),
\]
we first argue that necessarily 
\beq
d_F(q^j,\pmi)\to 0\;\mbox{as}\;j\to 0.\label{tinybubbles}
\eeq

Otherwise there would exist a subsequence (still denoted by $\e_j$) and a positive number $a_0$ such that along this subsequence the length
$E(\gamma_{\e_j})$ would exceed the value $E(\gamma_0)$ by at least $a_0$ and the addition of a small bump to $\gamma_0$ would create a competitor that beats $\gamma_{\e_j}$.

We then consider the smallest positive number $C_j$ such that
\[
\ell_j\subset \big\{p\in \R^2:\,\frac{1}{2}\lm_1p_1^2+\lm_2p_2^2\leq C_j\big\},
\]
with \eqref{tinybubbles} now implying that $C_j\to 0$.
Recalling Corollary \ref{ellipse}, we observe that the elliptical boundary of this set is in fact a closed isoperimetric curve in the metric $d_{F_H}$ corresponding to angle $\beta=0$. Consequently it is critical for the constrained minimization of $E_H$ and so the ellipse must satisfy \eqref{Euler} with conformal factor $F_H(p)
:=\sqrt{\lm_1^2p_1^2+\lm_2^2p_2^2}$ for some constant $\lm$. In fact, in the case of minimizing ellipses enclosing area $A$, we are in a setting where the (explicit) solution clearly varies smoothly with $A$ so that the heuristics presented at the outset of this section are in fact rigorous. Invoking \eqref{dEdA}
and the conclusion of Corollary \ref{ellipse} we then can identify the value of $\lm$ as $\lm_1+\lm_2$. Note in particular that the Lagrange multiplier is independent of $j$. 
 Expressing the ellipse as a parametric curve  $g_j:[a,b]\to\R^2$ it necessarily satisfies
 \eqref{Euler} and so when we take the inner product of this system of ODE's with the vector $(g_j')^\perp$ we obtain the relation 
\beq
F(g_j)\kappa(g_j)-\frac{1}{|g_j'|}\nabla F(g_j)\cdot(g_j')^\perp=\lm_1+\lm_2,\label{ellipseODE}
\eeq
where we have introduced notation for the curvature of $g_j$ with respect to the Euclidean metric, i.e.
\[
\kappa(g_j):=\frac{1}{|g_j'|^3}\,g_j''\cdot(g_j')^\perp.
\]
Now each of the closed loops $\ell_j$ is of course also a critical point solving \eqref{Euler} with its own corresponding Lagrange multiplier, say $\lm_j$,
so by a similar manipulation one has
\beq
F(\gamma_{\e_j})\kappa(\gamma_{\e_j})-\frac{1}{|\gamma_{\e_j}'|}\nabla F(\gamma_{\e_j})\cdot(\gamma_{\e_j}')^\perp=\lm_j,\label{loopODE}
\eeq
Since the contact between the ellipse $g_j$ and the loop $\gamma_{\e_j}$ must be tangential at $q_j$ with the loop lying inside the ellipse,
it follows that at $q_j$ we have the relations
\[
g_j=\gamma_{\e_j},\quad \frac{(g_j')^\perp}{|g_j'|}=\frac{(\gamma_{\e_j}')^\perp}{|\gamma_{\e_j}'|}\quad
\mbox{and}\quad \kappa(\gamma_{\e_j})\geq \kappa(g_j).
\]
Therefore, subtracting \eqref{ellipseODE} from \eqref{loopODE} we conclude that
\beq
\lm_j\geq \lm_1+\lm_2.\label{toobig}
\eeq
Now if in addition we have an upper bound on $\{\lm_j\}$ that is independent of $j$ then the contradiction is reached via compactness of $\{\gamma_{\e_j}\}$ exactly as was done
for the case of \eqref{segment} where now \eqref{toobig} is used in place of \eqref{gotcha}. 

Lastly, there remains the possibility that $\lm_j\to \infty$. In this case, we again focus on an arc, say $\zeta_j$, of $\gamma_{\e_j}$ that must span the the topological annulus $\mathcal{D}_r$ defined in \eqref{topannulus}. Working in a parametrization by Euclidean arclength  so that $|\zeta_j'|=1$ we then conclude via \eqref{loopODE} that necessarily the (Euclidean) curvature $\kappa(\zeta_j)$ at every point of $\zeta_j$ must tend to infinity with $\lm_j.$
 Consequently, the Euclidean arclength of $\gamma_{\e_j}$ in the annulus becomes infinitely long as well. Since in $\mathcal{D}_r$, the
conformal factor $F(\zeta_j)$ is bounded uniformly away from zero, this forces the value of $E(\gamma_{\e_j})$ to infinity, obviously contradicting its minimality.
\end{proof}

Theorem \ref{nobubbles} rules out the appearance of bubbles for the solutions obtained via Theorem \ref{mainresult} when the constraint value is sufficiently close to $\mathcal{A}(\gamma_0)$. For particular choices of the conformal factor $F$, however, such as the homogeneous choice of Section 3, we know that solutions may exist for large values of $A$ as well. It seems quite likely that one can find examples of $F$ for which solutions possessing bubbles exist at large values of the constraint.

\section{Application to bi-stable Hamiltonian traveling waves}
We conclude by recalling one of our motivations in considering the degenerate isoperimetric problem \eqref{mainprob}: the construction
of traveling wave solutions to the bi-stable Hamiltonian system
\beq
\J u_t=\Delta u -\nabla W_u(u)\quad\mbox{for}\;u:\R^n\times\R\to\R^2\label{RPBC}
\eeq
where  $\J$ denotes the symplectic matrix given by
\[
\J=\left(\begin{matrix} 0 & 1\\ -1& 0\end{matrix}\right).
\]
Here $W:\R^2\to\R$ is a potential such as the one considered in this article but whereas the results from the previous sections apply with an arbitrary numbers of zeros, here we will insist that $W$ vanishes at exactly two points.
By a traveling wave solution to \eqref{RPBC}  we mean a solution taking the form
\[u(x,t)=U(x_1-\nu t)=U(y)\quad\mbox{for some wave speed}\;\nu\in\R\]
 that joins the two minima of $W$, that is, where $U(\pm\infty)={\bf p}_{\pm}.$
Such a $U:\R\to\R^2$ must satisfy the system of ODE's
\beq
-\nu \J U'=U''-\nabla_uW(U)\quad\mbox{for}\;-\infty<y<\infty,\quad U(\pm\infty)={\bf p}_{\pm}.\label{ODE}
\eeq
One checks that \eqref{ODE} is the criticality condition associated with the constrained variational problem
\beq
\inf_{\mG_{A}}H(u)\quad\mbox{where}\quad H(u):=\int_{-\infty}^\infty \frac{1}{2}\abs{u'}^2 + W(u)\,dy\label{mu1}
\eeq
and where for a given $A\in\R$ the admissible set $\mG_{A}$ is defined by
\[
\mG_{A}:=\left\{u:\R\to\R^2:\;u-g\in H^1(\R;\R^2),\;\mathcal{A}(u)=A\right\}.
\]
In the above definition, $g:\R\to\R^2$ is any smooth function such that $g(y)\equiv \ppl$ for $y$ sufficiently
large and $g(y)\equiv \pmi$ for $y$ sufficiently negative. The wave speed $\nu$ arises as the Lagrange multiplier associated with the constraint 
$\mathcal{A}(u)=A$ in making this connection.

Applying the trivial inequality $H(u)\geq \sqrt{2}E(u)$ for $E$ given by \eqref{Edefn} and exploiting the convenient ``equi-partition"
parametrization $\abs{\gamma'}=\sqrt{2W(\gamma)}$ in competitors for \eqref{mainprob} one can establish that any isoperimetric curve $\gamma$ solving \eqref{mainprob} 
is automatically a minimizer of \eqref{mu1}, hence a solution to \eqref{ODE}, provided this minimizer has no `bubbles'; that is, provided
that scenario \eqref{loop} does not occur for either $\pmi$ or $\ppl.$ This is the content of \cite{ABCDS}, Theorem 4.2. In bridging these two seemingly disparate problems the wave speed $\nu$ is precisely the Lagrange multiplier $\lm=\lm(A)$ associated with the minimizer $\gamma.$
In light of Remark \ref{gc} this also means the wave speed can be characterized by the
relation $\nu=F^2(\gamma)\kappa_g$ where $\kappa_g$ is the geodesic curvature of $\gamma$.

In \cite{ABCDS} solutions to \eqref{mainprob} without bubbles are obtained for constraint values near $\mathcal{A}(\gamma_0)$ under the
assumption that the potential $W$ is identically equal to a non-degenerate quadratic in neighborhoods of $\pmi$ and $\ppl$. 
In light of our existence result in Theorem \ref{mainresult} as well as Theorem \ref{nobubbles} that rules out bubbles, we can now assert the corresponding result for more general potentials $W$:
\bthm\label{tw}
Assume as in Theorem \ref{mainresult} that $W:\R^2\to\R$ is a smooth function vanishing only at points $\pmi$ and $\ppl$ in $\R^2$ with $D^2W({\bf{p}}_{\pm})>0$. Let $\gamma_0$ denote an unconstrained planar geodesic joining $\pmi$ to $\ppl$ as in Theorem \ref{ungeo}. Then
there exists a value $\e_1>0$ such that for all $A$ in the interval $\big(\mathcal{A}(\gamma_0)-\e_1,\mathcal{A}(\gamma_0)+\e_1\big)$, the curves, say $\gamma_A$, described in Theorem \ref{mainresult} and given an equipartition parametrization, will solve \eqref{mu1} and therefore solve
 \eqref{ODE} as well with a wave speed $\nu_A$ depending on $A$.
\ethm

As noted in \cite{ABCDS}, without further assumptions there remains the possibility of a degeneracy of the curve $\gamma_0$, parametrized so that
$\abs{\gamma_0'}=\sqrt{2W(\gamma_0)}$, with respect to the energy
$H$ which could lead to the non-generic predicament that $\nu_A=0$ for all $A$ in the interval of existence. In this situation the traveling waves
would all in fact be standing waves, that is, simply a one-parameter family of heteroclinic connections solving
\beq
U''-\nabla_uW(U)=0\quad\mbox{for}\;-\infty<y<\infty,\quad U(\pm\infty)={\bf p}_{\pm}. \label{hetero}
\eeq
One way of eliminating this scenario is to add a non-degeneracy assumption regarding the second variation operator
\[
 \delta^2H(\gamma)[\Phi] := \int_{-\infty}^\infty
       \left[ |\Phi'|^2 + \Phi\cdot D^2 W(U_0)\Phi\right] dy,
       \]
       acting on $\Phi\in H^1(\R;\R^2).$
Here we quote Proposition 4.4 from \cite{ABCDS}:
\bprop\label{simple}
In addition to the assumptions of Theorem \ref{tw}, assume that there exist constants $R_0,c_0>0$ such that
\[\nabla W(p)\cdot p\ge  c_0  |p|^2
\quad \text{for all $p\in\R^2$ with} \quad |p|\ge R_0.
\]
Assume also that for every solution $u_0$ to \eqref{hetero} that minimizes $H$ among competitors satisfying $u(\pm\infty)={\bf p}_\pm$ one has that
zero is an isolated simple eigenvalue of $\delta^2 H(u_0)$.
  Then, for any  such $H$-minimizing heteroclinic, there exists $\e_2>0$ such that
 if $A\in (\mathcal{A}(\gamma_0)-\e_2, \mathcal{A}(\gamma_0)+\e_2)$ with $A\not=\mathcal{A}(u_0)$ and if $u_A:\R\to\R^2$ solves \eqref{mu1},  then $u_A$ solves \eqref{ODE} with wave speed $\nu=\nu_A\neq 0$.
 \eprop
We point out that the function $u_0'$ will always be an eigenfunction of $\delta^2 H(u_0)$ with corresponding eigenvalue zero but this proposition asserts that if $u_0'$ is the only such eigenfunction and if zero is isolated in the spectrum  for any $H$-minimizing heteroclinic, then the curves produced in Theorem \ref{mainresult}, appropriately parametrized, will constitute traveling waves that definitely travel.

\bibliographystyle{plain}

\end{document}